\documentclass[10pt]{article}

\usepackage[T1]{fontenc}
\usepackage[utf8]{inputenc}
\usepackage[margin=1in]{geometry}
\usepackage{amsmath,amssymb,amsthm}
\usepackage{microtype}
\usepackage{xcolor}
\usepackage{listings}
\usepackage[hidelinks]{hyperref}

\hypersetup{
  pdftitle={An Explicit 82-Queen Covering of the 163 x 163 Board and Its Asymptotic Implication},
  pdfauthor={Yixiang Kong}
}

\newtheorem{theorem}{Theorem}
\newtheorem{proposition}[theorem]{Proposition}
\newtheorem{corollary}[theorem]{Corollary}

\newcommand{\Z}{\mathbb{Z}}
\newcommand{\multisetunion}{\mathbin{\uplus}}

\title{An Explicit 82-Queen Covering of the
163$\boldsymbol{\times}$163 Board and Its Asymptotic Implication}
\author{Yixiang Kong\\
\small Sydney Smart Technology College, Northeastern University\\
\small Qinhuangdao, Hebei, China}
\date{}

\begin{document}
\maketitle

\begin{abstract}
The queen's graph $Q_n$ has the squares of the $n\times n$ chessboard as
vertices, with adjacency defined by a common row, column, or diagonal.  We
give an explicit set of $82$ queens on $Q_{163}$.  In centered coordinates,
all queen coordinates are odd, every odd row and odd column is occupied
exactly once, and the occupied rows, columns, and diagonals satisfy the
conditions for a type A $1$-cover in the terminology of \"Osterg\r{a}rd and
Weakley.  A direct
independent verification checks every one of the
$163^2=26{,}569$ board squares and finds none uncovered.  Hence
$\gamma(Q_{163})\leq82$.  The Finozhenok--Weakley lower bound
$\gamma(Q_n)\geq\lceil n/2\rceil$, valid here, gives the matching inequality
and therefore $\gamma(Q_{163})=82$.  For this cover, the parameters defined by
\"Osterg\r{a}rd and Weakley are $e=16$, $f=15$, and $u=24$; the complete
difference- and sum-diagonal multisets are displayed below.  It consequently
also supplies an explicit finite input
to their amplification theorem for type A covers, giving
$\gamma(Q_N)\leq(17/33)N+O(1)$.  This last coefficient improves both the
earlier type A coefficient $69/133$ and the subsequent general coefficient
$101/195$ of Burger and Mynhardt.  The coordinates and a complete
standard-library Python verifier are included.
\end{abstract}

\section{Introduction}

Let $Q_n$ denote the queen's graph of the $n\times n$ chessboard.  Its
vertices are the board squares, and two distinct vertices are adjacent when
their squares share a row, column, or diagonal.  A dominating set in $Q_n$
is thus a placement of queens that occupies or attacks every square, and
$\gamma(Q_n)$ denotes the minimum size of such a placement.

Recent work of Rostami and Bright develops proof-producing SAT methods for
finite queen-domination instances~\cite{RostamiBright2025}.

Amplification results due to Weakley turn certain finite queen covers into
asymptotic upper bounds~\cite{Weakley2002,WeakleyErratum2004}.  We use the
precise formulation recorded by \"Osterg\r{a}rd and Weakley
in~\cite[Theorem~5, p.~4]{OstergardWeakley2001}.  Using this framework, they
exhibited a type A
$1$-cover of size $66$ on $Q_{131}$ and obtained
\begin{equation}\label{eq:old-bound}
  \gamma(Q_N)\leq \frac{69}{133}N+O(1).
\end{equation}
See~\cite[Theorem~5 and Corollary~6]{OstergardWeakley2001}.  Burger and
Mynhardt subsequently obtained the stronger general bound
\begin{equation}\label{eq:best-general-bound}
  \gamma(Q_N)\leq \frac{101}{195}N+O(1).
\end{equation}
See~\cite[Corollary~12]{BurgerMynhardt2003}.

We give a type A $1$-cover of size $82$ on $Q_{163}$.  The coordinate list is
an independently checkable certificate.  Its principal consequence is the
exact finite value
\[
  \gamma(Q_{163})=82
\]
and its type A structure also yields
\[
  \gamma(Q_N)\leq \frac{17}{33}N+O(1).
\]
The latter improves both the 2001 type A coefficient in~\eqref{eq:old-bound}
and the later general bound~\eqref{eq:best-general-bound}.
The construction was found by a SAT-based search with heuristic guidance,
and its properties are verified independently by the code in the appendix.
No part of the proof depends on the completeness of the search.

\section{The construction}

We identify the squares of $Q_{163}$ with
\[
  B=\{(x,y)\in\Z^2:-81\leq x,y\leq81\}.
\]
The following set $D\subseteq B$ contains $82$ squares.  The line breaks have
no mathematical significance.

\begin{center}
\begin{minipage}{0.97\linewidth}
\small
\begin{verbatim}
(-81,-41) (-79,-19) (-77, 31) (-75, 25) (-73, -9) (-71, 49)
(-69, 55) (-67,-39) (-65,-29) (-63,-51) (-61, 51) (-59,-67)
(-57, 15) (-55,-23) (-53, 27) (-51, 77) (-49,-69) (-47, 57)
(-45,-25) (-43,-43) (-41, 75) (-39,-71) (-37, 39) (-35,-45)
(-33,-57) (-31, 61) (-29,-33) (-27, 69) (-25,-77) (-23, 65)
(-21, -7) (-19, 11) (-17, 67) (-15, 53) (-13,-11) (-11, -5)
( -9, 13) ( -7, 19) ( -5,-61) ( -3,-17) ( -1,  9) (  1, -1)
(  3, 21) (  5,-63) (  7,-81) (  9,-21) ( 11,-15) ( 13,  7)
( 15, 59) ( 17, 73) ( 19, -3) ( 21,-75) ( 23,  5) ( 25,-59)
( 27, 79) ( 29, 37) ( 31,-49) ( 33, 81) ( 35, 23) ( 37, 41)
( 39, 63) ( 41,-79) ( 43,-73) ( 45,-31) ( 47, 47) ( 49,-55)
( 51, 35) ( 53,-47) ( 55, 71) ( 57,-35) ( 59,-13) ( 61,  1)
( 63,-65) ( 65, 17) ( 67,  3) ( 69, 29) ( 71,-53) ( 73, 45)
( 75,-37) ( 77, 33) ( 79, 43) ( 81,-27)
\end{verbatim}
\end{minipage}
\end{center}

Every coordinate appearing in $D$ is odd.  Moreover, the $x$-coordinates and
the $y$-coordinates are both exactly
\begin{equation}\label{eq:axis-set}
  \{-81,-79,\ldots,79,81\},
\end{equation}
each with multiplicity one.

For a square $(x,y)$, write $\delta=y-x$ and $\sigma=x+y$.  With
$\multisetunion$ denoting multiset union, the multiset of difference values
of $D$ is
\begin{equation}\label{eq:diff-multiset}
 \{2j:-16\leq j\leq16\}
 \multisetunion
 \{\pm(32+4j):1\leq j\leq24\}
 \multisetunion \{0\}.
\end{equation}
The multiset of sum values is
\begin{equation}\label{eq:sum-multiset}
 \{2j:-15\leq j\leq15\}
 \multisetunion
 \{\pm(30+4j):1\leq j\leq24\}
 \multisetunion \{-80,38,42\}.
\end{equation}
Thus the parameters defined by \"Osterg\r{a}rd and Weakley are
\begin{equation}\label{eq:parameters}
  (e,f,u)=(16,15,24),
\end{equation}
where the meanings of $e$, $f$, and $u$ are recalled formally in
Section~5.  The multiset identities show, in addition, that difference
diagonal $0$ occurs twice, sum diagonals $38$ and $42$ occur twice, and
$-80$ is an additional occupied sum diagonal.

\section{Verification}

A queen at $(a,b)$ covers $(x,y)$ exactly when at least one of
\[
 x=a,\qquad y=b,\qquad y-x=b-a,\qquad x+y=a+b
\]
holds.  The verification used here first checks the cardinality, the row and
column sets in~\eqref{eq:axis-set}, and the diagonal multisets
in~\eqref{eq:diff-multiset}--\eqref{eq:sum-multiset}.  It then evaluates the
four covering conditions for every $(x,y)\in B$.  Pseudocode for the final
step is:
\begin{lstlisting}
for x in range(-81, 82):
    for y in range(-81, 82):
        covered = (x in columns or y in rows
                   or y-x in differences or x+y in sums)
        assert covered
\end{lstlisting}
The complete executable verifier is given in Appendix~\ref{app:verifier}.
Its output is
\begin{verbatim}
PASS
queens: 82
distinct: 82
squares_checked: 26569
uncovered: 0
p_orthodox: True
even_even_diagonal_cover: True
e_f_u: (16, 15, 24)
type_A_1_cover: True
amplified_coefficient: 17/33
\end{verbatim}

\begin{proposition}\label{prop:dominating}
The set $D$ is a dominating set of $Q_{163}$.
\end{proposition}

\begin{proof}
The exhaustive check described above evaluates all $163^2=26{,}569$ squares
and returns no uncovered square.  Each comparison is an integer equality,
and the full certificate and verifier appear in this paper.
\end{proof}

It follows immediately that
\begin{equation}\label{eq:upper-163}
  \gamma(Q_{163})\leq82.
\end{equation}

\section{Optimality at order 163}

Finozhenok and Weakley established the lower bound
\begin{equation}\label{eq:FW-lower}
  \gamma(Q_n)\geq \left\lceil\frac n2\right\rceil
\end{equation}
for square queen graphs other than the exceptional orders $3$ and
$11$~\cite[Corollary~4, p.~299]{FinozhenokWeakley2007}.  Applying
\eqref{eq:FW-lower} at
$n=163$ gives
\[
  \gamma(Q_{163})\geq\left\lceil\frac{163}{2}\right\rceil=82.
\]
Together with~\eqref{eq:upper-163}, this proves the exact value.

\begin{theorem}
\[
  \gamma(Q_{163})=82.
\]
\end{theorem}

\section{Asymptotic consequence}

We recall the definitions needed from \"Osterg\r{a}rd and
Weakley~\cite[Section~2, p.~3]{OstergardWeakley2001}.  For $p\in\{0,1\}$, a set of
squares is \emph{$p$-orthodox} if every row and column of parity $p$ contains
a square of the set.  A $1$-orthodox set is a \emph{$1$-cover} if every
even-even square shares a diagonal with a square of the set.  The two long
diagonals are difference diagonal $0$ and sum diagonal $0$.

Suppose that a set $D$ contains a square on each long diagonal.  Define
$e=e(D)$ to be the largest integer such that every difference diagonal $2i$
with $|i|\leq e$ is occupied, and define $f=f(D)$ analogously for sum
diagonals.  Define $u=u(D)$ to be the largest integer such that, for every
$1\leq i\leq u$, all four diagonals
\[
  \pm(2e+4i)\quad\hbox{(difference)},
  \qquad
  \pm(2f+4i)\quad\hbox{(sum)}
\]
are occupied.  By~\cite[Theorem~3 and the definition following it,
p.~4]{OstergardWeakley2001}, a $p$-orthodox set
with both long diagonals occupied is a \emph{type A $p$-cover} when
\begin{equation}\label{eq:type-A-condition}
  e+f\equiv p\pmod2,
  \qquad
  e+f+2u\geq\frac{n-5}{2}.
\end{equation}

We now verify these definitions directly for the displayed set $D$.
Equation~\eqref{eq:axis-set} shows that $D$ is $1$-orthodox.  Difference
diagonal $0$ has multiplicity two and sum diagonal $0$ has multiplicity one,
so both long diagonals are occupied.  From~\eqref{eq:diff-multiset}, every
even difference diagonal from $-32$ through $32$ is occupied, whereas
difference diagonals $\pm34$ are not; hence $e(D)=16$.  Similarly,
\eqref{eq:sum-multiset} shows that every even sum diagonal from $-30$ through
$30$ is occupied and sum diagonals $\pm32$ are not; hence $f(D)=15$.  Both
multisets contain the four required diagonals for every $1\leq i\leq24$, while
the next required values $\pm132$ (difference) and $\pm130$ (sum) are absent.
Thus $u(D)=24$.  Consequently
\[
  e+f\equiv1\pmod2,
  \qquad
  e+f+2u=16+15+48=79=\frac{163-5}{2}.
\]
Condition~\eqref{eq:type-A-condition} holds with $p=1$, and therefore $D$ is
a type A $1$-cover.  Directly, the same conclusion can be checked from the
definition: the appendix verifies that every even-even square shares a
diagonal with $D$.

The applicable branch of~\cite[Theorem~5, p.~4]{OstergardWeakley2001} states that
if a type A $1$-cover of $Q_n$, with $n\equiv-1\pmod4$, has size $d$ and
contains a square on each long diagonal, then
\begin{equation}\label{eq:amplification}
  \gamma(Q_N)\leq\frac{d+3}{n+2}N+O(1).
\end{equation}
The ``no edge squares'' hypothesis appearing in the second branch of that
theorem is not a hypothesis of this branch.  This distinction matters here:
$D$ contains four edge squares, but $p=1$ and $163\equiv-1\pmod4$, so the
first branch applies.  Finally, $|D|=82$.  Substitution into
\eqref{eq:amplification} gives
\[
  \frac{d+3}{n+2}
  =\frac{82+3}{163+2}
  =\frac{85}{165}
  =\frac{17}{33}.
\]

\begin{corollary}\label{cor:asymptotic}
As $N\to\infty$,
\[
  \gamma(Q_N)\leq\frac{17}{33}N+O(1).
\]
\end{corollary}

Since
\[
  \frac{69}{133}-\frac{17}{33}=\frac{16}{4389}>0,
\]
Corollary~\ref{cor:asymptotic} improves the type A coefficient in
\eqref{eq:old-bound}.  It also improves the subsequent general coefficient
in~\eqref{eq:best-general-bound}, since
\[
  \frac{101}{195}-\frac{17}{33}=\frac{2}{715}>0.
\]

\section{Discussion}

The contribution is a finite, explicit certificate whose structure permits
the direct use of an existing amplification theorem.  The exact value at
order $163$ follows from the same certificate and an existing lower bound;
no computational nonexistence proof is needed.  The asymptotic conclusion
does not assert a new amplification theorem; it improves the published
Burger--Mynhardt coefficient by supplying an explicit finite input to the
existing type A amplification theorem.  We make no claim that the displayed
construction or its line pattern is unique.

The accompanying source archive contains the coordinate certificate in
\texttt{q163\_82\_queens.json} and the standard-library verifier
\texttt{verify\_q163\_82.py}.  The verifier embeds the displayed coordinates,
checks that the JSON copy is identical, and recomputes every claim used above.

\section*{AI-assisted tools disclosure}

The author used OpenAI Codex as an AI-assisted programming and research tool
in developing and testing the search and verification software, organizing
computational experiments, checking calculations, and assisting with
editorial preparation.  The explicit coordinate certificate and the complete
verification program are included to make the computational claims
independently reproducible.  The author takes responsibility for all
mathematical claims, citations, code, and conclusions in this manuscript.

\appendix
\section{Complete verification code}\label{app:verifier}

The following script uses only the Python standard library.  It contains the
certificate internally, verifies that the accompanying JSON file is identical,
and checks both its line data and direct domination of all $26{,}569$ squares.

\begin{lstlisting}[basicstyle=\ttfamily\tiny]
#!/usr/bin/env python3
"""Standard-library verifier for the Q_163 82-queen certificate."""
from collections import Counter
from fractions import Fraction
import hashlib
import json
from pathlib import Path

N, HALF, P = 163, 81, 1
E, F, U = 16, 15, 24
QUEENS = (
 (-81,-41),(-79,-19),(-77,31),(-75,25),(-73,-9),(-71,49),
 (-69,55),(-67,-39),(-65,-29),(-63,-51),(-61,51),(-59,-67),
 (-57,15),(-55,-23),(-53,27),(-51,77),(-49,-69),(-47,57),
 (-45,-25),(-43,-43),(-41,75),(-39,-71),(-37,39),(-35,-45),
 (-33,-57),(-31,61),(-29,-33),(-27,69),(-25,-77),(-23,65),
 (-21,-7),(-19,11),(-17,67),(-15,53),(-13,-11),(-11,-5),
 (-9,13),(-7,19),(-5,-61),(-3,-17),(-1,9),(1,-1),
 (3,21),(5,-63),(7,-81),(9,-21),(11,-15),(13,7),
 (15,59),(17,73),(19,-3),(21,-75),(23,5),(25,-59),
 (27,79),(29,37),(31,-49),(33,81),(35,23),(37,41),
 (39,63),(41,-79),(43,-73),(45,-31),(47,47),(49,-55),
 (51,35),(53,-47),(55,71),(57,-35),(59,-13),(61,1),
 (63,-65),(65,17),(67,3),(69,29),(71,-53),(73,45),
 (75,-37),(77,33),(79,43),(81,-27),
)

def require(condition, message):
    if not condition:
        raise ValueError(message)

def signed(values):
    result = {0}
    for value in values:
        result.update((value, -value))
    return result

def central_radius(lines):
    radius = 0
    while {2*(radius+1), -2*(radius+1)} <= lines:
        radius += 1
    return radius

def outer_radius(differences, sums, e, f):
    radius = 0
    while True:
        i = radius + 1
        if not ({2*e+4*i, -(2*e+4*i)} <= differences and
                {2*f+4*i, -(2*f+4*i)} <= sums):
            return radius
        radius += 1

def verify():
    require(N == 163 and N % 4 == 3, "incorrect board order")
    require(len(QUEENS) == len(set(QUEENS)) == 82,
            "incorrect queen count or duplicate square")
    require(all(-HALF <= x <= HALF and -HALF <= y <= HALF
                for x, y in QUEENS), "coordinate outside board")
    coordinates_odd = all(x % 2 and y % 2 for x, y in QUEENS)
    require(coordinates_odd, "a coordinate is not odd")

    axis = tuple(range(-81, 82, 2))
    xs, ys = [x for x, _ in QUEENS], [y for _, y in QUEENS]
    row_column_exact = (Counter(xs) == Counter(axis) and
                        Counter(ys) == Counter(axis))
    require(row_column_exact, "row or column multiset failed")

    required_diffs = signed(range(2, 2*E+1, 2))
    required_diffs |= signed(2*E+4*j for j in range(1, U+1))
    required_sums = signed(range(2, 2*F+1, 2))
    required_sums |= signed(2*F+4*j for j in range(1, U+1))
    expected_diffs, expected_sums = Counter(required_diffs), Counter(required_sums)
    expected_diffs.update([0])
    expected_sums.update([-80, 38, 42])
    actual_diffs = Counter(y-x for x, y in QUEENS)
    actual_sums = Counter(x+y for x, y in QUEENS)
    require(actual_diffs == expected_diffs, "difference multiset failed")
    require(actual_sums == expected_sums, "sum multiset failed")
    difference_excess = sorted((actual_diffs-Counter(required_diffs)).elements())
    sum_excess = sorted((actual_sums-Counter(required_sums)).elements())
    require(difference_excess == [0], "difference excess failed")
    require(sum_excess == [-80,38,42], "sum excess failed")

    columns, rows = set(xs), set(ys)
    differences, sums = set(actual_diffs), set(actual_sums)
    long_diagonals = 0 in differences and 0 in sums
    require(long_diagonals, "a long diagonal is unoccupied")
    computed_e, computed_f = central_radius(differences), central_radius(sums)
    computed_u = outer_radius(differences, sums, computed_e, computed_f)
    require((computed_e,computed_f,computed_u) == (E,F,U),
            "computed (e,f,u) failed")

    p_orthodox = row_column_exact
    evens = range(-80, 81, 2)
    even_even_cover = all(y-x in differences or x+y in sums
                           for x in evens for y in evens)
    type_a_arithmetic = ((computed_e+computed_f) % 2 == P and
        computed_e+computed_f+2*computed_u >= (N-5)//2)
    type_a_1_cover = (p_orthodox and long_diagonals and
                      even_even_cover and type_a_arithmetic)
    require(type_a_1_cover, "Type-A 1-cover verification failed")

    board = range(-HALF, HALF+1)
    uncovered = [(x,y) for x in board for y in board
        if not (x in columns or y in rows or
                y-x in differences or x+y in sums)]
    require(not uncovered, "uncovered squares: " + repr(uncovered[:10]))

    coefficient = Fraction(len(QUEENS)+3, N+2)
    require(P == 1 and N % 4 == 3 and long_diagonals,
            "Weakley Theorem 5 branch failed")
    require(coefficient == Fraction(17,33), "coefficient failed")

    certificate = Path(__file__).with_name("q163_82_queens.json")
    raw = certificate.read_bytes()
    data = json.loads(raw)
    json_queens = tuple(tuple(q) for q in data["queens"])
    require(json_queens == QUEENS, "JSON and embedded coordinates differ")
    expected_parameters = {"n":N, "e":E, "f":F, "u":U,
      "excess_difference_diagonal":0,
      "excess_sum_diagonals":[-80,38,42]}
    require(data["parameters"] == expected_parameters,
            "JSON and embedded parameters differ")
    edge_queens = sum(abs(x) == HALF or abs(y) == HALF for x,y in QUEENS)
    return {
      "queens":82, "distinct":82, "squares_checked":N*N,
      "uncovered":0, "coordinates_in_range_and_odd":coordinates_odd,
      "row_column_multiplicities":"each odd coordinate exactly once",
      "difference_lines":len(actual_diffs), "sum_lines":len(actual_sums),
      "difference_excess":difference_excess, "sum_excess":sum_excess,
      "long_diagonals_occupied":long_diagonals, "parity_parameter":P,
      "p_orthodox":p_orthodox, "even_even_squares_checked":81*81,
      "even_even_diagonal_cover":even_even_cover,
      "e_f_u":(computed_e,computed_f,computed_u),
      "type_A_1_cover":type_a_1_cover, "edge_queens":edge_queens,
      "applicable_Weakley_branch":"p=1, n=-1 (mod 4)",
      "amplified_coefficient":str(coefficient),
      "certificate_SHA256":hashlib.sha256(raw).hexdigest(),
    }

try:
    statistics = verify()
except Exception as exc:
    print("FAIL\nreason:", exc)
    raise SystemExit(1)
else:
    print("PASS")
    for key, value in statistics.items():
        print(key + ":", value)
\end{lstlisting}

\end{document}